\documentclass[10pt,leqno]{amsart}
\usepackage{amsmath}
\usepackage{amsfonts}
\usepackage{amssymb}
\usepackage{graphicx}
\usepackage{color}
\usepackage{hyperref}
\usepackage{xcolor}
\newtheorem{theorem}{Theorem}[section]

\newtheorem*{theorem*}{Theorem}
\newtheorem{lemma}{Lemma}[section]

\newtheorem{proposition}{Proposition}[section]

\newtheorem{conjecture}[theorem]{Conjecture}
\theoremstyle{definition}
\newtheorem{definition}[theorem]{Definition}
\newtheorem{example}[theorem]{Example}

\newtheorem{remark}[theorem]{Remark}

\newcommand{\R}{\mathbb{R}}

\def \b {\beta}

\def\Ric{\text{Ric}}

\def\a{\alpha}
\def\l{\lambda}

\def\R{\mathbb{R}}

\def\vp{\varphi}

\def\id{\operatorname{id}}
\def\Ric{\operatorname{Ric}}

\def\tr{\operatorname{tr}}

\numberwithin{equation}{section}
\makeatletter
\newcommand*\owedge{\mathpalette\@owedge\relax}
\newcommand*\@owedge[1]{%
  \mathbin{%
    \ooalign{%
      $#1\m@th\bigcirc$\cr
      \hidewidth$#1\m@th\wedge$\hidewidth\cr
    }%
  }%
}
\makeatother

\DeclareMathOperator{\spann}{span}

\begin{document}

\title[The curvature operator of the second kind]{Manifolds with nonnegative curvature operator of the second kind}

\author{Xiaolong Li}\thanks{The author's research is partially supported by Simons Collaboration Grant \#962228 and a start-up grant at Wichita State University}
\address{Department of Mathematics, Statistics and Physics, Wichita State University, Wichita, KS, 67260}
\email{xiaolong.li@wichita.edu}

\subjclass[2020]{53C20, 53C21, 53C24}

\keywords{Curvature operator of the second kind, Nishikawa's conjecture, differentiable sphere theorem, rigidity theorem}

\begin{abstract}
We investigate the curvature operator of the second kind on Riemannian manifolds and prove several classification results. The first one asserts that a closed Riemannian manifold with three-positive curvature operator of the second kind is diffeomorphic to a spherical space form, improving a recent result of Cao-Gursky-Tran assuming two-positivity. 
The second one states that a closed Riemannian manifold with three-nonnegative curvature operator of the second kind is either diffeomorphic to a spherical space form, or flat, or isometric to a quotient of a compact irreducible symmetric space. This settles the nonnegativity part of Nishikawa's conjecture under a weaker assumption.
\end{abstract}

\maketitle

\section{Introduction}

Riemannian manifolds with positive curvature operator have been of great interest in geometry and topology for a long time, and a number of remarkable results have been obtained by various authors. For instance, Meyer \cite{Meyer71} proved, using the Bochner technique, that a closed Riemannian manifold with positive curvature operator must be a real homology sphere. 
A well-known theorem of Tachibana \cite{Tachibana74} states that an Einstein manifold with positive curvature operator must have constant positive sectional curvature. 
Indeed, these two results are both consequences of the following celebrated differentiable sphere theorem. 
\begin{theorem}\label{thm Bohm-Wilking}
A closed Riemannian manifold with two-positive curvature operator is diffeomorphic to a spherical space form.
\end{theorem}

Recall that a Riemannian manifold $(M^n,g)$ is said to have two-positive curvature operator if the sum of the smallest two eigenvalues of the curvature operator $\hat{R}_p: \Lambda^2(T_pM) \to  \Lambda^2(T_pM)$ defined via \begin{equation}\label{eq 1.0}
    \hat{R}(e_i\wedge e_j) =\frac 1 2 \sum_{k,l}R_{ijkl}e_k \wedge e_l
\end{equation}
is positive for any $p\in M$. Here $\Lambda^2(T_pM)$ is the space of two-forms and $\{e_1, \cdots, e_n\}$ is an orthonormal basis of $T_pM$. 
In dimension three, two-positive curvature operator is equivalent to positive Ricci curvature, so Theorem \ref{thm Bohm-Wilking} is due to Hamilton \cite{Hamilton82}, who introduced Ricci flow and used it to prove that three-manifolds with positive Ricci curvature are diffeomorphic to spherical space forms.  
Hamilton \cite{hamilton86} also showed the above differentiable sphere theorem for closed four-manifolds with positive curvature operator, which was extended to two-positive curvature operator by Chen \cite{Chen91}. In higher dimensions, Theorem \ref{thm Bohm-Wilking} is due to B\"ohm and Wilking \cite{BW08}. 
The proof, in all dimensions, uses the normalized Ricci flow to evolve an initial metric with two-positive curvature operator into a limit metric with constant positive sectional curvature. See also \cite{BS08}, \cite{BS09}, \cite{Brendle08}, \cite{AN09}, \cite{NWolfson07}, \cite{NW10} and \cite{BS11} for differentiable sphere theorems under other curvature conditions. 

Rigidity results hold if the curvature operator is nonnegative or two-nonnegative. Gallot and Meyer \cite{GM75} proved that a closed Riemannian manifold with nonnegative curvature operator is either reducible, or locally symmetric, or its universal cover has the cohomology of a sphere or a complex projective space.
Tachibana \cite{Tachibana74} proved that closed Einstein manifolds with nonnegative curvature operator must be locally symmetric. 
More generally, the following rigidity result was obtained by Hamilton \cite{hamilton86} in dimension three, Hamilton \cite{hamilton86} and Chen \cite{Chen91} in dimension four, and Ni and Wu \cite{NWu07} in all higher dimensions (see also Wilking's wonderful survey \cite[Theorem 1.13]{Wilking07}).

\begin{theorem}\label{thm Ni-Wu}
A closed simply-connected Riemannian manifold $(M^n,g)$ with two-nonnegative curvature operator satisfies one of the following statements.
\begin{enumerate}
    \item $M$ is diffeomorphic to $\mathbb{S}^n$;
    \item $n=2m$ and $M$ is a K\"ahler manifold biholomorphic to $\mathbb{CP}^m$;
    \item $M$ is isometric to a compact irreducible symmetric space;
    \item $M$ is isometric to nontrivial Riemannian product.
\end{enumerate}
\end{theorem}

Next, let's recall the second kind of curvature operator.
The Riemann curvature tensor $R_{ijkl}$ also acts naturally on the space of symmetric two-tensors $S^2(T_pM)$ (see for example \cite{CV60}, \cite{BE69} and \cite{BK78}). 
This action, denoted by $\mathring{R}:S^2(T_pM) \to S^2(T_pM)$ in this paper, is defined by 
\begin{equation*}
    \mathring{R}(e_i \odot e_j) =\sum_{k,l}R_{iklj} e_k \odot e_l,
\end{equation*}
where $\odot$ denotes the symmetric product.
The new feature here is that $S^2(T_pM)$ is not irreducible under the action of the orthogonal group $O(T_pM)$. As in \cite{Nishikawa86}, one should consider the induced symmetric bilinear form 
\begin{equation}\label{eq 1.1}
  \mathring{R}:S^2_0(T_pM) \times S^2_0(T_pM) \to \R
\end{equation}
by restricting $\mathring{R}$ to the space of traceless symmetric two-tensors $S^2_0(T_pM)$.  
Following Nishikawa's terminology in \cite{Nishikawa86}, we call the symmetric bilinear form $\mathring{R}$ in \eqref{eq 1.1} the \textit{curvature operator of the second kind}, to distinguish it from the curvature operator $\hat{R}$ defined in \eqref{eq 1.0}, which he called the \textit{curvature operator of the first kind}. 

The action of Riemann curvature tensor on symmetric two-tensors indeed has a long history. It appeared for K\"ahler manifolds in the study of deformation of complex analytic structures by Calabi and Vesentini \cite{CV60}. They introduced the self-adjoint operator $\xi_{\a \b} \to R^{\rho}_{\ \a\b}{}^{\sigma} \xi_{\rho \sigma}$ from $S^2(T^{1,0}_p M)$ to itself, and computed the eigenvalues of this operator on Hermitian symmetric spaces of classical type, with the exceptional ones handled shortly afterward by Borel \cite{Borel60}. In the Riemannian setting, the operator $\mathring{R}$ arises naturally in the context of deformations of Einstein structure in Berger and Ebin \cite{BE69} (see also \cite{Koiso79a, Koiso79b} and \cite{Besse08}). 
In addition, it appears naturally in the Bochner-Weitzenb\"ock formulas for symmetric two-tensors (see for example \cite{MRS20}), for differential forms in \cite{OT79} and for Riemannian curvature tensors in \cite{Kashiwada93}.  
In another direction, curvature pinching estimates for $\mathring{R}$ were studied by Bourguignon and Karcher \cite{BK78}, and they also calculated eigenvalues of $\mathring{R}$ on the complex projective space with the Fubini-Study metric and the quaternionic projective space with its canonical metric.

Given the above-mentioned beautiful theorems for the curvature operator $\hat{R}$, it is an intriguing question to classify closed Riemannian manifolds with positive or nonnegative curvature operator of the second kind.
Indeed, Nishikawa \cite[Conjecture II]{Nishikawa86} proposed the following conjecture in 1986. 
\begin{conjecture}[Nishikawa \cite{Nishikawa86}]
Let $(M^n,g)$ be a closed Riemannian manifold. 
\begin{enumerate}
    \item If $M$ has positive curvature operator of the second kind, then $M$ is diffeomorphic to a spherical space form.
    \item If $M$ has nonnegative curvature operator of the second kind, then $M$ is diffeomorphic to a Riemannian locally symmetric space.
\end{enumerate}
\end{conjecture}

The conjecture was made based on two known results. One asserts that $M$ must be a real homology sphere if it has positive curvature operator of the second kind, which is a result of Ogiue and Tachibana \cite{OT79}. The other one says the conjecture holds if $M$ is either Einstein or has harmonic curvature tensor, which is a result of Kashiwada \cite{Kashiwada93}.

To the best of the author's knowledge, no progress was made on Nishikawa's conjecture until the recent work of Cao, Gursky and Tran \cite{CGT21}, in which they proved the positive case of Nishikawa's conjecture under a weaker assumption. 

\begin{theorem}[Cao-Gursky-Tran \cite{CGT21}]\label{thm Cao etc}
A closed Riemannian manifold with two-positive curvature operator of the second kind is diffeomorphic to a spherical space form. 
\end{theorem}

Their key observation is that two-positive curvature operator of the second kind implies the strictly PIC1 condition introduced by Brendle \cite{Brendle08} (see also Definition \ref{def curvature conditions}). 
The positive case of Nishikawa's conjecture follows immediately from Brendle's result in \cite{Brendle08} asserting that the normalized Ricci flow evolves an initial metric satisfying strictly PIC1 into a limit metric with constant positive sectional curvature.

The first purpose of this paper is to prove an improvement of Theorem \ref{thm Cao etc} by weakening the assumption to three-positivity.  The second purpose is to settle the nonnegativity part of Nishikawa's conjecture.

In dimension two, it is easy to see that both positivity and two-positivity (respectively, nonnegativity and two-nonnegativity) of the curvature operator of the second kind are equivalent to positive (respectively, nonnegative) scalar curvature. Thus, Nishikawa's conjecture is a consequence of the uniformization theorem. 
So we begin with dimension three.

\begin{theorem}\label{thm 3d}\footnote{An improvement has been obtained by the author in a subsequent work \cite{Li22JGA} assuming $3\frac 1 3$-positive/nonnegative curvature operator of the second kind.}
Let $(M^3,g)$ be a closed Riemannian manifold of dimension three.
\begin{enumerate}
    \item If $M$ has three-positive curvature operator of the second kind, then $M$ is diffeomorphic to a spherical space form. 
    \item If $M$ has three-nonnegative curvature operator of the second kind, then $M$ is either flat or diffeomorphic to a spherical space form.
\end{enumerate}
\end{theorem}

To prove this theorem, we show in Proposition \ref{prop 3d} that in dimension three,  three-positive (respectively, three-nonnegative) curvature operator of the second kind implies $\Ric > \frac{S}{12} >0$ (respectively, $\Ric \geq  \frac{S}{12} \geq 0$), where $S$ denotes the scalar curvature.
Part (1) of Theorem \ref{thm 3d} then follows from Hamilton's classification of closed three-manifolds with positive Ricci curvature. To prove part (2) of Theorem \ref{thm 3d}, we first observe that the manifold must be locally irreducible if it is not flat, and then apply Hamilton's classification of closed three-manifolds with nonnegative Ricci curvature \cite{hamilton86}. 

We would like to point out that the curvature assumption in Theorem \ref{thm 3d} is optimal, in the sense that the theorem fails if one only assumes four-positive or four-nonnegative curvature operator of the second kind. In particular, the product manifold $\mathbb{S}^2 \times \mathbb{S}^1$ has four-positive curvature operator of the second kind, but it does not have three-nonnegative curvature operator of the second kind (see Example \ref{eg cylinder}). 

In dimensions four and above, we prove that
\begin{theorem}\label{thm 4+}
A closed Riemannian manifold with three-positive curvature operator of the second kind is diffeomorphic to a spherical space form.
\end{theorem}
It was shown in \cite{CGT21} that two-positive curvature operator of the second kind implies strictly PIC1. We strengthen their result by showing that three-positive curvature operator of the second kind implies strictly PIC1 (see Proposition \ref{prop algebraic implications}). Theorem \ref{thm 4+} then follows from this improvement and Brendle's result \cite{Brendle08}. 

The corresponding rigidity result states
\begin{theorem}\label{thm Nishikawa}
Let $(M^n,g)$ be a closed Riemannian manifold of dimension $n\geq 4$. If $M$ has three-nonnegative curvature operator of the second kind, then one of the following statements holds: 
\begin{enumerate}
    \item $(M,g)$ is flat;
    \item $M$ is diffeomorphic to a spherical space form;
    \item $n\geq 5$ and the universal cover of $M$ is isometric to a compact irreducible symmetric space \footnote{This possibility has been ruled out by Nienhaus, Petersen and Wink \cite{NPW22}, as they proved that a closed $n$-dimensional Riemannian manifold with $ \frac{n+2}{2}$-nonnegative curvature operator of the second kind is either a rational homology sphere or flat.}.
\end{enumerate}
\end{theorem}

Theorem \ref{thm Nishikawa} settles the nonnegativity part of Nishikawa's conjecture under the weaker assumption of three-nonnegative curvature operator of the second kind. 
The key idea is to reduce the problem to the locally irreducible case and make use of the classification of closed locally irreducible Riemannian manifolds with weakly PIC1. For this purpose, we show that 

\begin{theorem}\label{thm irreducible}
Let $(M^n,g)$ be an $n$-dimensional non-flat complete Riemannian manifold with $n$-nonnegative curvature operator of the second kind. Then $M$ is locally irreducible \footnote{Some improvements of this result have been obtained in \cite{NPWW22} and \cite{Li22product}.}.
\end{theorem}

Note that the assumption here cannot be weakened to $(n+1)$-nonnegative curvature operator of the second kind, as the product manifold $\mathbb{S}^{n-1}\times \mathbb{S}^1$ has $(n+1)$-nonnegative curvature operator but it is reducible (see Example \ref{eg cylinder}).

It is well-known that nonnegative curvature operator preserves product manifolds, in the sense that if $M$ is isometric to $M_1 \times M_2$, then $M$ has nonnegative curvature operator if and only if both $M_1$ and $M_2$ have nonnegative curvature operator. However, Theorem \ref{thm irreducible} implies that the product manifold $M_1 \times M_2$ has $n$-nonnegative curvature operator of the second kind if and only if both $M_1$ and $M_2$ are flat.

It follows from the calculation in \cite{BK78} that the complex projective space $\mathbb{CP}^m$ with the Fubini-Study metric does not have four-nonnegative curvature operator of the second kind. We prove a more general result here. 
\begin{theorem}\label{thm kahler}
A K\"ahler manifold with four-nonnegative curvature operator of the second kind is flat \footnote{The curvature operator of the second kind on K\"ahler manifolds has been further investigated in \cite{Li22PAMS, Li22Kahler} and \cite{NPWW22}.}. 
\end{theorem} 

Compact Hermitian symmetric spaces are examples of K\"ahler manifolds with nonnegative curvature operator, but none of them has four-nonnegative curvature operator of the second kind, according to the above theorem.
We also point out that the assumption in Theorem \ref{thm kahler} cannot be weakened to five-nonnegative curvature operator of the second kind in general, as $\mathbb{CP}^2$ has five-positive curvature operator of the second kind (see Example \ref{eg complex projective space}).

The paper is organized as follows. 
In Section 2, we give an introduction to the curvature operator of the second kind, state some conventions and definitions, and collect some examples on which the eigenvalues of the curvature operator of the second kind can be computed explicitly. 
In Section 3, we examine the curvature operator of the second kind in dimension three and prove Theorem \ref{thm 3d}. 
Section 4 is devoted to establishing various algebraic relations between the curvature operator of the second kind and other frequently used curvature conditions such as sectional curvature, Ricci curvature and isotropic curvature. 
In Section 5, we study the curvature operator of the second kind on product manifolds and prove Theorem \ref{thm irreducible}.
In Section 6, we prove Theorem \ref{thm kahler} under a slightly weaker assumption. 
Finally, the proof of Theorem \ref{thm Nishikawa} is presented in Section 7.

\section{The Curvature Operator of the second kind}

We give an introduction to the curvature operator of the second kind in this section, and the reader is referred to \cite{CV60, OT79, BK78, Nishikawa86, Kashiwada93, CGT21} for more information and previous results.

Let $(V,g)$ be a Euclidean vector space of dimension $n\geq 2$.  We always identify $V$ with its dual space $V^*$ via the metric $g$. 
The space of bilinear forms on $V$ is denote by $T^2(V)$, and it splits as 
\begin{equation*}
    T^2(V)=S^2(V)\oplus  \Lambda^2(V),
\end{equation*}
where $S^2(V)$ is the space of symmetric two-tensors on $V$ and $\Lambda^2(V)$ is the space of two-forms on $V$.
Our conventions on symmetric products and wedge products are that, for $u$ and $v$ in $V$, $\odot$ denotes the symmetric product defined by 
\begin{equation*}
    u \odot v  =u \otimes v + v \otimes u,
\end{equation*}
and 
$\wedge$ denotes the wedge product defined by 
\begin{equation*}
    u \wedge v  =u \otimes v -v \otimes u.
\end{equation*}
The inner product $g$ on $V$ naturally induces inner products on $S^2(V)$ and $\Lambda^2(V)$, respectively. 
To be consistent with \cite{CGT21}, the inner product on $S^2(V)$ is defined as 
\begin{equation*}
    \langle A, B \rangle =\tr(A^T B),
\end{equation*}
and the inner product on $\Lambda^2(V)$ is defined as 
\begin{equation*}
    \langle A, B \rangle =\frac{1}{2}\tr(A^T B).
\end{equation*}
In particular, if $\{e_1, \cdots, e_n\}$ is an orthonormal basis for $V$, then 
$\{e_i \wedge e_j\}_{1\leq i < j \leq n}$ is an orthonormal basis for $\Lambda^2(V)$ and $\{\frac{1}{\sqrt{2}}e_i \odot e_j\}_{1\leq i < j \leq n} \cup \{\frac{1}{2} e_i \odot e_i\}_{1\leq i\leq n}$ is an orthonormal basis for $S^2(V)$.

The space of symmetric two-tensors on $\Lambda^2(V)$ has the orthogonal decomposition 
\begin{equation*}
    S^2(\Lambda^2 (V))=S^2_B(\Lambda^2 (V)) \oplus \Lambda^4 (V),
\end{equation*}
where $S^2_B(\Lambda^2 (V))$ consists of all tensors $R\in S^2(\Lambda^2 (V))$ that also satisfy the first Bianchi identity. 
$S^2_B(\Lambda^2 (V))$ is called the space of algebraic curvature operators. 

By the symmetries of $R\in  S^2_B(\Lambda^2 (V))$ (not including the first Bianchi identity), there are (up to sign) two ways that $R$ can induce a symmetric linear map $R:T^2(V) \to T^2(V)$. 
The first one, denoted by $\hat{R}: \Lambda^2(V) \to \Lambda^2(V)$ in this paper, is the so-called curvature operator defined by
\begin{equation}\label{eq R hat}
    \hat{R}(e_i\wedge e_j) =\frac 1 2 \sum_{k,l}R_{ijkl}e_k \wedge e_l,
\end{equation}
where $\{e_1, \cdots, e_n\}$ is an orthonormal basis of $V$. 
Note that if the eigenvalues of $\hat{R}$ are all greater than or equal to $\kappa \in \R$, then all the sectional curvatures of $R$ are bounded from below by $\kappa$.  
We say $R\in S^2_B(\Lambda^2 (V))$ has positive (respectively, nonnegative) curvature operator if all the eigenvalues of $\hat{R}$ are positive (respectively, nonnegative). 
For $2\leq k \leq \binom{n}{2}$, we say $R\in S^2_B(\Lambda^2 (V))$ has $k$-positive (respectively, $k$-nonnegative) curvature operator if the sum of the smallest $k$ eigenvalues of $\hat{R}$ is positive (respectively, nonnegative). 
We say a Riemannian manifold has positive (respectively, nonnegative, $k$-positive,  $k$-nonnegative) curvature operator if the Riemann curvature tensor at each point has positive (respectively, nonnegative, $k$-positive, $k$-nonnegative) curvature operator.

The second one, denoted by $\mathring{R}:S^2(V) \to S^2(V)$, is defined by 
\begin{equation}\label{eq R ring}
    \mathring{R}(e_i \odot e_j) =\sum_{k,l}R_{iklj} e_k \odot e_l.
\end{equation}
However, on contrary to the case of $\hat{R}$, all eigenvalues of $\mathring{R}$ being nonnegative implies all the sectional curvatures of $R$ are zero, that it, $R\equiv 0$ \footnote{This follows from the observation that the trace of $\mathring{R}:S^2(V) \to S^2(V)$ is equal to $\frac{S}{2}$ and $\mathring{R}(g,g)=-S$.}. 
The new feature here is that $S^2(V)$ is not irreducible under the action of the orthogonal group $O(V)$ of $V$. 
The space $S^2(V)$ splits into $O(V)$-irreducible subspaces as 
\begin{equation*}
    S^2(V)=S^2_0(V) \oplus \R g,
\end{equation*}
where $S^2_0(V)$ denotes the space of traceless symmetric two-tensors on $V$.
The map $\mathring{R}$ defined in \eqref{eq R ring} then induces a symmetric bilinear form
\begin{equation}\label{eq 2.3}
    \mathring{R}:S^2_0(V) \times S^2_0(V) \to \R
\end{equation}
by restriction to $S^2_0(V)$. 
Note that if all the eigenvalues of $\mathring{R}$ restricted to $S^2_0(V)$ are bounded from below by $\kappa \in \R$, then the sectional curvatures of $R$ are bounded from below by $\kappa$. 
It should be noted that $\mathring{R}$ does not preserve the subspace $S^2_0(V)$ in general, but it does, for instance, when $R$ is Einstein.

Following \cite{Nishikawa86}, we call the symmetric bilinear form $\mathring{R}$ in \eqref{eq 2.3} the \textit{curvature operator of the second kind}, to distinguish it from the map $\hat{R}$ defined in \eqref{eq R hat}, which he called the \textit{curvature operator of the first kind}.

\begin{remark}
It was pointed out in \cite{NPW22} that the curvature operator of the second kind can also be interpreted as the self-adjoint operator $\pi \circ \mathring{R}:S^2_0(V) \to S^2_0(V)$, where $\pi$ is the projection map from $S^2(V)$ onto $S^2_0(V)$. This is equivalent to the interpretation as the symmetric bilinear form in \eqref{eq 2.3}, as 
\begin{equation*}
    \mathring{R}(\vp,\psi)=\langle \mathring{R}(\vp), \psi\rangle =\langle \pi \circ \mathring{R} (\vp), \psi \rangle=  (\pi \circ \mathring{R})(\vp,\psi)
\end{equation*}
for any $\vp, \psi \in S^2_0(V)$.  
\end{remark}

Let $N=\dim(S^2_0(V))=\frac{(n-1)(n+2)}{2}$ and $\{\vp_i\}_{i=1}^N$ be an orthonormal basis of $S^2_0(V)$. The $N\times N$ matrix $\mathring{R}(\vp_i, \vp_j)$ is called the matrix representation of $\mathring{R}$ with respect to the orthonormal basis $\{\vp_i\}_{i=1}^N$. 
The eigenvalues of $\mathring{R}$ refers to the eigenvalues of any of its matrix representation. 
This is independent of the choices of the orthonormal bases because matrix representations of $\mathring{R}$ with respect to different orthonormal bases of $S^2_0(V)$ are similar to each other.
We then make the following definitions. 
\begin{definition}
    For a positive integer $1\leq k \leq N$, we say $R\in S^2_B(\Lambda^2(V))$ has $k$-nonnegative (respectively, $k$-positive) curvature operator of the second kind if the sum of the smallest $k$-eigenvalues of $\mathring{R}$ is nonnegative (respectively, positive). 
\end{definition}
\begin{definition}
    A Riemannian manifold $(M^n,g)$ is said to have $k$-nonnegative (respectively, $k$-positive) curvature operator of the second kind if $R_p \in S^2_B(\Lambda^2 T_pM)$ has $k$-nonnegative (respectively, $k$-positive) curvature operator of the second kind for each $p\in M$.  
\end{definition}

We conclude this section by collecting some examples on which the eigenvalues of $\mathring{R}$ are known. These examples are used to demonstrate the sharpness of many results in this paper.

\begin{example}\label{eg sphere}
Let $(\mathbb{S}^n,g_0)$ be the $n$-sphere with constant sectional curvature $1$. 
Its Riemann curvature tensor is given by $R_{ijkl}=g_{ik}g_{jl}-g_{il}g_{jk}$.
For any traceless symmetric two-tensor $\vp$, we have 
\begin{align*}
    \mathring{R}(\vp,\vp)&=\sum_{i,j,k,l=1}^n R_{ijkl}\vp_{il}\vp_{jk} \\
    & =\sum_{i,j,k,l=1}^n (g_{ik}g_{jl}-g_{il}g_{jk})\vp_{il}\vp_{jk} \\
    & =|\vp|^2 -\tr(\vp)^2 \\
    & = |\vp|^2.
\end{align*}
Thus $\mathring{R}$ is equal to the identity map on $S^2_0(T_p\mathbb{S}^n)$ at any point $p\in \mathbb{S}^n$ and all eigenvalues of $\mathring{R}$ restricted to $S^2_0(T_p\mathbb{S}^n)$ are equal to $1$. 
\end{example}

\begin{example}\label{eg complex projective space}
Let $(\mathbb{CP}^m,g_{FS})$ be the complex projective space of complex dimension $m$ with the Fubini-Study metric $g_{FS}$. Then 
$\mathring{R}$ restricted to $S^2_0(T_p\mathbb{CP}^m)$ has two distinct eigenvalues: $-2$ with multiplicity $(m-1)(m+1)$ and $4$ with multiplicity $m(m+1)$. These eigenvalues, together with their corresponding eigenspaces, are determined in \cite{BK78}. 
In particular, $\mathbb{CP}^2$ has five-positive (but not four-positive) curvature operator of the second kind.
\end{example}

\begin{example}\label{eg cylinder}
Let $M=\mathbb{S}^{n-1} \times \R$, with $\mathbb{S}^{n-1}$ being the $(n-1)$-sphere with constant sectional curvature $1$. 
The eigenvalues of $\mathring{R}$ restricted to $S^2_0$ are given by 
\begin{align*}
    \l_1 &=-\frac{n-2}{n},   \text{ with multiplicity } 1; \\
    \l_2 &=0,  \text{ with multiplicity } n-1; \\
    \l_3 &=1,  \text{ with multiplicity } \frac{(n-2)(n+1)}{2}.
\end{align*}
The eigenvector associated with $\l_1$ is given by
$\sum_{i=1}^{n-1} e_i \odot e_i -(n-1)e_n \odot e_n,$ where $\{e_1,\cdots,e_{n-1}\}$ is an orthonormal basis of $T_p\mathbb{S}^{n-1}$ and $e_n$ is a unit vector in $T_q \R$.
The eigenspace of $\l_2$ has dimension $n-1$ and it is spanned by vectors of the form $u \odot v$ with $u \in T_p\mathbb{S}^{n-1}$ and $v\in T_q \R$. 
The eigenspace of $\l_3$ is the space of traceless symmetric two-tensors on the $\mathbb{S}^{n-1}$ factor, whose dimension is $\frac{(n-2)(n+1)}{2}$. 
In particular, $\mathbb{S}^{n-1} \times \R$ has $(n+1)$-positive (but not $n$-positive) curvature operator of the second kind.
\end{example}

\begin{example}\label{eg product of spheres}
Let $M=\mathbb{S}^{2} \times \mathbb{S}^{2}$ with $\mathbb{S}^{2}$ being the $2$-sphere with constant sectional curvature $1$. 
The eigenvalues of $\mathring{R}$ are given by $\l_1=-1$ with multiplicity one, $\l_2=0$ with multiplicity $4$ and $\l_3=1$ with multiplicity $4$ (see \cite{CGT21}). If we pick an orthonormal basis $\{e_1, e_2, e_3, e_4\}$ of $T_pM$ with $e_1,e_2$ in the tangent space of the first $\mathbb{S}^2$ factor and $e_3,e_4$ in the tangent space of the second $\mathbb{S}^2$ factor, then the corresponding eigenspaces are given by
\begin{eqnarray*}
E_1 &=& \spann \{e_1 \odot e_1 +e_2 \odot e_2 -e_3 \odot e_3 -e_4 \odot e_4 \}, \\
E_2 &=& \spann \{e_1\odot e_3, e_1\odot e_4, e_2\odot e_3, e_2\odot e_4 \},  \\
E_3 &=& \spann \{e_1 \odot e_2, e_3 \odot e_4, e_1 \odot e_1 -e_2 \odot e_2, e_3 \odot e_3- e_4 \odot e_4  \}.
\end{eqnarray*}
In particular, $\mathbb{S}^{2} \times \mathbb{S}^{2}$ has six-nonnegative (but not five-nonnegative) curvature operator of the second kind.
\end{example}

\section{Dimension Three}

In this section, we investigate the curvature operator of the second kind in dimension three. 

Let $(V,g)$ be a three-dimensional Euclidean vector space and let $\{e_1,e_2,e_3\}$ be an orthonormal basis for $V$. 
The space of traceless symmetric two-tensors $S^2_0(V)$ has dimension $5$, and we can choose an orthonormal basis for it as follows.
\begin{eqnarray*}
\vp_1 &=& \frac{1}{\sqrt{2}} e_1 \odot e_2, \\
\vp_2 &=& \frac{1}{\sqrt{2}} e_1 \odot e_3, \\
\vp_3 &=& \frac{1}{\sqrt{2}} e_2 \odot e_3, \\
\vp_4 &=& \frac{1}{2\sqrt{2}} \left( e_1 \odot e_1- e_2 \odot e_2 \right) , \\
\vp_5 &=& \frac{1}{2\sqrt{6}} \left( e_1 \odot e_1+e_2 \odot e_2 -2e_3 \odot e_3 \right) . \\
\end{eqnarray*}

\begin{lemma}\label{lemma 3.1}
For the $\vp_i$'s defined as above, we have the following identities.
\begin{eqnarray*}
\mathring{R}(\vp_1,\vp_1) &=& R_{1212}, \\
\mathring{R}(\vp_2,\vp_2) &=& R_{1313}, \\
\mathring{R}(\vp_3,\vp_3) &=& R_{2323}, \\
\mathring{R}(\vp_4,\vp_4) &=& R_{1212}, \\
\mathring{R}(\vp_5,\vp_5) &=& \frac 2 3 (R_{1313}+R_{2323}) -\frac 1 3 R_{1212}. 
\end{eqnarray*}
\end{lemma}
\begin{proof}
It follows from direct calculation.
For instance, the only nontrivial components of $\vp_5$ are $(\vp_5)_{11}=(\vp_5)_{22}=\frac{1}{\sqrt{6}}$ and $(\vp_5)_{33}=-\frac{2}{\sqrt{6}}$, and we calculate that
\begin{eqnarray*}
\mathring{R}(\vp_5,\vp_5) &=& \sum_{i,j,k,l=1}^3 R_{ijkl}(\vp_5)_{il}(\vp_5)_{jk}\\
&=& R_{1221}(\vp_5)_{11}(\vp_5)_{22} +R_{1331}(\vp_5)_{11}(\vp_5)_{33} \\
&& +R_{2112}(\vp_5)_{22}(\vp_5)_{11}
+R_{2332}(\vp_5)_{22}(\vp_5)_{33} \\
&& +R_{3113}(\vp_5)_{33}(\vp_5)_{11}
+R_{3223}(\vp_5)_{33}(\vp_5)_{22} \\
&=& \frac 2 3 (R_{1313}+R_{2323}) -\frac 1 3 R_{1212}.
\end{eqnarray*}
The other ones are similar. 
\end{proof}
%It was observed in \cite{CGT21} that nonnegative curvature operator of the second kind implies $\Ric \geq 
%\frac{S}{6} \geq 0$. 

One immediately reads from Lemma \ref{lemma 3.1} that two-positive (respectively, two-nonnegative) curvature operator of the second kind implies positive (respectively, nonnegative) sectional curvature, as 
\begin{equation*}
    2R_{1212} = \mathring{R}(\vp_1,\vp_1)+\mathring{R}(\vp_4,\vp_4).
\end{equation*} 
This observation in fact remains valid in all dimensions (see Proposition \ref{prop algebraic implications}).

Another consequence of Lemma \ref{lemma 3.1}, first pointed out in \cite{CGT21}, is that positive (respectively, nonnegative) curvature operator of the second kind implies that $\Ric > \frac{S}{6} >0$ (respectively, $\Ric \geq  \frac{S}{6} \geq 0$). This follows from 
\begin{eqnarray*}
\mathring{R}(\vp_5,\vp_5) &=& \frac 2 3 (R_{1313}+R_{2323}) -\frac 1 3 R_{1212} \\
&=& \frac{2}{3} R_{33} -\frac{1}{3}\left(\frac{S}{2}-R_{33}\right) \\
&=& R_{33}-\frac{S}{6}.
\end{eqnarray*}

Similarly, the following proposition can be easily derived from Lemma \ref{lemma 3.1}. 
\begin{proposition}\label{prop 3d}
Let $R \in S^2_B (\Lambda^2 V)$ be an algebraic curvature operator on a three-dimensional Euclidean vector space $V$. Denote by $S$ the scalar curvature of $R$. 
\begin{enumerate}
    \item If $R$ has three-positive curvature operator of the second kind, then $$\Ric > \frac{S}{12} >0.$$
    \item If $R$ has three-nonnegative curvature operator of the second kind, then 
    $$\Ric \geq \frac{S}{12} \geq 0.$$ Moreover, if the Ricci curvature has an eigenvalue zero, then $R\equiv 0$.  
\end{enumerate}
\end{proposition}

\begin{proof}
(1). If $R$ has three-positive curvature operator of the second kind, then 
\begin{eqnarray*}
0& < & \mathring{R}(\vp_2,\vp_2)+\mathring{R}(\vp_3,\vp_3)+\mathring{R}(\vp_5,\vp_5) \\
&=& \frac 5 3 (R_{1313}+R_{2323}) -\frac 1 3 R_{1212} \\
&=& \frac 5 3 R_{33} -\frac 1 3 \left( \frac{S}{2}-R_{33}\right)\\
&=& 2 \left( R_{33} -\frac{S}{12} \right).
\end{eqnarray*}
Thus we have $\Ric > \frac{S}{12}$. Tracing it yields $S > \frac{S}{4}$, which implies $S>0$.

(2). The first part follows similarly as in part (1).
If Ricci curvature has an eigenvalue zero, then $S=0$. In view of $\Ric \geq 0$, we must have $\Ric \equiv 0$ and thus $R \equiv 0$.
\end{proof}

We are ready to prove Theorem \ref{thm 3d}.

\begin{proof}[Proof of Theorem \ref{thm 3d}]
(1). By part (1) of Proposition \ref{prop 3d}, $M$ has positive Ricci curvature. The conclusion follows from Hamilton's classification of closed three-manifolds with positive Ricci curvature in \cite{Hamilton82}.

(2). By part (2) of Proposition \ref{prop 3d}, $M$ has nonnegative Ricci curvature. 
If $M$ is not flat, then it must be locally irreducible. Otherwise, the universal cover of $M$ splits as $N^2\times \R$, whose Ricci curvature has a zero eigenvalue everywhere. Since $N^2 \times \R$ also has three-nonnegative curvature operator of the second kind, it must be flat by part (2) of Proposition \ref{prop 3d}. 
The desired conclusion then follows from Hamilton's classification of closed locally irreducible three-manifolds with nonnegative Ricci curvature \cite{hamilton86}.  
\end{proof}

\section{Algebraic Implications}

In this section, we investigate the curvature operator of the second kind in dimensions four and above, and establish various algebraic relations with other frequently used curvature conditions such as sectional curvature and isotropic curvature. 
For the reader's convenience, we first recall some definitions regarding isotropic curvature and its variants.

\begin{definition}\label{def curvature conditions}
Let $R \in S^2_B (\Lambda^2 V)$ be an algebraic curvature operator on a Euclidean vector space $V$ of dimension $n\geq 4$. 
\begin{enumerate}
    \item We say $R$ has nonnegative isotropic curvature if for all orthonormal four-frame $\{e_1, e_2, e_3, e_4\}\subset V$, it holds that 
    \begin{equation*}
        R_{1313}+R_{1414}+R_{2323}+R_{2424}-2R_{1234} \geq 0.
    \end{equation*}
    If the inequality is strict, $R$ is said to have positive isotropic curvature. 
    \item We say $R$ has weakly PIC1 if for all orthonormal four-frame $\{e_1, e_2, e_3, e_4\}\subset V$ and all $\l \in [-1,1]$, it holds that 
    \begin{equation*}
        R_{1313}+\l^2 R_{1414}+R_{2323}+\l^2 R_{2424}-2\l R_{1234} \geq 0.
    \end{equation*}
    If the inequality is strict, $R$ is said to have strictly PIC1. 
    \item We say $R$ has nonnegative complex sectional curvature (or weakly PIC2) if for all orthonormal four-frame $\{e_1, e_2, e_3, e_4\}\subset V$ and all $\l, \mu \in [-1,1]$, it holds that 
    \begin{equation*}
        R_{1313}+\l^2 R_{1414}+\mu^2 R_{2323}+\l^2 \mu^2  R_{2424}-2\l \mu R_{1234} \geq 0.
    \end{equation*}
    If the inequality is strict, $R$ is said to have positive complex sectional curvature (or strictly PIC2). 
\end{enumerate}
\end{definition}

The notion of positive isotropic curvature was introduced by Micallef and Moore \cite{MM88} in their study of minimal two-spheres in Riemannian manifolds. They proved that a simply connected closed Riemannian manifold with positive isotropic curvature is homeomorphic to the sphere. 
The PIC2 curvature condition was introduced by Brendle and Schoen and it played a central role in their proof of the quarter-pinched differentiable sphere theorem \cite{BS09}. 
Ni and Wolfson \cite{NWolfson07} discovered that strictly (respectively, weakly) PIC2 is equivalent to positive (respectively, nonnegative) complex sectional curvature. They also gave a simple proof that positive and nonnegative complex sectional curvature are preserved by the Ricci flow. 
The PIC1 curvature condition was introduced by Brendle \cite{Brendle08}.
All the above curvature conditions in Definition \ref{def curvature conditions} are preserved by the Ricci flow (see for example \cite{Wilking13} for a unified simple proof). Moreover, the normalized Ricci flow evolves an initial metric with strictly PIC1 (or the stronger condition strictly PIC2) into a limit metric with constant positive sectional curvature (see \cite{BS09, Brendle08, NWolfson07}). 

Next, we summarize the algebraic relations between the curvature operator of the second kind and other curvature conditions.
\begin{proposition}\label{prop algebraic implications}
Let $R \in S^2_B (\Lambda^2 V)$ be an algebraic curvature operator on a Euclidean vector space $V$ of dimension $n\geq 4$. 
\begin{itemize}
    \item [(1)] $R$ has $\frac{(n-1)(n+2)}{2}$-positive (respectively, $\frac{(n-1)(n+2)}{2}$-nonnegative) curvature operator of the second kind if and only if $R$ has positive (respectively, nonnegative) scalar curvature. 
    \item [(2)] If $R$ has $n$-positive (respectively, $n$-nonnegative) curvature operator of the second kind, then 
    $\Ric > \frac{S}{n(n+1)} > 0$ (respectively, $\Ric \geq \frac{S}{n(n+1)} \geq 0$ ).
    \item [(3)] If $R$ has two-positive (respectively, two-nonnegative) curvature operator of the second kind, then $R$ has positive (respectively, nonnegative) sectional curvature.
    \item [(4)] If $R$ has three-positive (respectively, three-nonnegative) curvature operator of the second kind, then $R$ has strictly (respectively, weakly) PIC1.
    \item [(5)] If $R$ has positive (respectively, nonnegative) curvature operator of the second kind, then $R$ has positive (respectively, nonnegative) complex sectional curvature.
\end{itemize}
\end{proposition}

\begin{proof}[Proof of Proposition \ref{prop algebraic implications}]
(1). Since the dimension of $S^2_0(V)$ is equal to $\frac{(n-1)(n+2)}{2}$,  $\mathring{R}$ being $\frac{(n-1)(n+2)}{2}$-positive (respectively, nonnegative) is equivalent to the trace of $\mathring{R}$ (restricted on $S^2_0(V)$) being positive (respectively, nonnegative). 
It can be easily seen that the total trace of $\mathring{R}$ on $S^2(V)$ is equal to $\frac{S}{2}$ and that $\mathring{R}$ restricted on the one-dimensional subspace $\R g$ is equal to multiplication by $-\frac{S}{n}$.
Thus the trace of $\mathring{R}$ restricted on $S^2_0(V)$ is equal to $\frac{n+2}{2n}S$.

(2). Let $\{e_1, \cdots, e_n\}$ be an orthonormal basis of $V$ and define
\begin{equation*}
    \vp_1 = \frac{1}{2\sqrt{n(n-1)}}\left((n-1)e_1 \odot e_1 -\sum_{j=2}^n e_j \odot e_j \right), 
\end{equation*}
and 
$    \vp_i = \frac{1}{\sqrt{2}} e_1 \odot e_{i} \text{ for } 2\leq i \leq n. $
Then $\{\vp_1, \cdots, \vp_n\}$ is an orthonormal subset of $S^2_0(V)$. 
Using the observation that the nonzero components of $\vp_1$ are 
$    (\vp_1)_{11} =\sqrt{\frac{n-1}{n}},$
and 
$  (\vp_1)_{jj} = -\frac{1}{\sqrt{n(n-1)}} \text{ for } 2\leq j\leq n$,  
we calculate that 
\begin{eqnarray*}
    \mathring{R}(\vp_1,\vp_1)
    &=& \sum_{i,j=1}^n R_{ijji}(\vp_1)_{ii}(\vp_1)_{jj} \\
    &=& \frac{2}{n}\sum_{j=2}^n R_{1j1j} -\frac{1}{n(n-1)}\sum_{i,j=2}^{n} R_{ijij} \\
    &=& \frac{2}{n}R_{11}-\frac{1}{n(n-1)}\left(S-2R_{11} \right) \\
    &=& \frac{2}{n-1}R_{11} -\frac{1}{n(n-1)}S.
\end{eqnarray*}
For $2\leq i \leq n$, we have
\begin{equation*}
    \mathring{R}(\vp_i,\vp_i) =R_{1i1i}.
\end{equation*}
Since $R$ has $n$-nonnegative curvature operator of the second kind, we get 
\begin{eqnarray*}
    0 & \leq& \mathring{R}(\vp_1,\vp_1)+\sum_{i=2}^n \mathring{R}(\vp_i,\vp_i) \\
    &=& \frac{2}{n-1}R_{11} -\frac{1}{n(n-1)}S +\sum_{i=2}^n R_{1i1i} \\
    &=& \frac{n+1}{n-1}R_{11} -\frac{1}{n(n-1)}S,
\end{eqnarray*}
which implies that $\Ric \geq \frac{S}{n(n+1)} \geq 0$.
Similarly, $n$-positive curvature operator of the second kind implies $\Ric > \frac{S}{n(n+1)} > 0$.

(3). Notice that
\begin{eqnarray*}
\vp_1 &=& \frac{1}{\sqrt{2}}e_1 \odot e_2, \\
\vp_2 &=& \frac{1}{2\sqrt{2}} \left( e_1 \odot e_1 -e_2 \odot e_2 \right)
\end{eqnarray*}
are orthonormal traceless symmetric two-tensors in $S^2_0(V)$, where $\{e_1,e_2\}$ is an orthonormal two-frame in $V$. 
In view of 
\begin{equation*}
    2R_{1212}=\mathring{R}(\vp_1,\vp_1)+\mathring{R}(\vp_2,\vp_2),
\end{equation*}
we see that two-positive (respectively, two-nonnegative) curvature operator of the second kind implies positive (respectively, nonnegative) sectional curvature. 

(4). This is a consequence of Lemma \ref{lemma PIC1} below.
For any orthonormal four-frame $\{e_1, \cdots, e_4\} \subset V$ and any $\l \in [-1,1]$, we have by Lemma \ref{lemma PIC1} that
\begin{eqnarray*}
    && R_{1313}+\l^2R_{1414}+R_{2323}+\l^2R_{2424} -4\l R_{1234} \geq 0, \\
    && R_{1313}+\l^2R_{1414}+R_{2323}+\l^2R_{2424} + 4\l R_{1234} \geq 0.
\end{eqnarray*}
Therefore, we have 
\begin{equation*}
    R_{1313}+\l^2R_{1414}+R_{2323}+\l^2R_{2424} \geq  4 | \l R_{1234}| \geq 2|\l R_{1234}|,
\end{equation*}
from which we conclude that $$R_{1313}+\l^2R_{1414}+R_{2323}+\l^2 R_{2424} -2\l R_{1234} \geq 0.$$ 
Hence $R$ has weakly PIC1.

(5). This follows from Lemma \ref{lemma PIC2} below and a similar argument as in (4).

\end{proof}

\begin{lemma}\label{lemma PIC1}
Suppose $n=\dim(V)\geq 4$ and $R \in S^2_B(\Lambda^2(V))$ has three-positive (respectively, three-nonnegative) curvature operator of the second kind.  
Then for any orthonormal four-frame
 $\{e_1, \cdots, e_4\} \subset V$ and any $\l \in \R$, it holds that 
\begin{equation*}%\label{eq 1}
    R_{1313}+\l^2 R_{1414}+R_{2323}+\l^2 R_{2424} - 4 \l  R_{1234} > \text{ (respectively, $\geq$) } 0.
\end{equation*}
\end{lemma}

\begin{proof}
We shall only prove the case that $R$ has three-nonnegative curvature operator of the second kind, as the three-positive case is similar.

Let $\{e_1, \cdots, e_4\}$ be an orthonormal four-frame in $V$ and $\l \in \R$. Define the following traceless symmetric two-tensors on $V$. 
\begin{eqnarray*}
\vp_1 &=& \frac 1 2 \left( e_1\odot e_1+\l e_2\odot e_2 -e_3\odot e_3 -\l e_4 \odot e_4\right) ,\\
\vp_2 &=&\l e_1 \odot e_4 -e_2\odot e_3 , \\
\vp_3 &=& e_1 \odot e_3 +\l e_2\odot e_4 . 
\end{eqnarray*}
One easily verifies that these tensors are mutually orthogonal in $S^2_0(V)$ and of the same magnitude $\sqrt{2(1+\l^2)}$. 

Next, we compute $\mathring{R}(\vp_i,\vp_i)$ for $i=1,2,3$. 
Notice that the only nonzero components of $\vp_1$ are 
\begin{equation*}
    (\vp_{1})_{11}=1, (\vp_{1})_{22}=\l, (\vp_{1})_{33}=-1, (\vp_{1})_{44}=-\l.
\end{equation*}
So we calculate 
\begin{eqnarray*}
    \mathring{R}(\vp_1,\vp_1) &=& \sum_{i,j,k,l=1}^n R_{ijkl}(\vp_1)_{il}(\vp_1)_{jk} \\
    &=& \l R_{1221} - R_{1331} -\l R_{1441} + \l R_{2112} -\l R_{2332} -\l^2 R_{2442} \\
    && -R_{3113} - \l R_{3223} +\l R_{3443} - \l R_{4114} -\l^2 R_{4224} +\l R_{4334}\\
    &=& 2(-\l R_{1212}-\l R_{3434}+R_{1313}+\l^2 R_{2424}+ \l R_{1414}+\l R_{2323}).
\end{eqnarray*}
Similarly,  
\begin{eqnarray*}
\mathring{R}(\vp_2,\vp_2) &=& 2(\l^2 R_{1414}+ R_{2323}-2\l R_{1234}+2\l R_{1342} ), \\
\mathring{R}(\vp_3,\vp_3) &=& 2(R_{1313}+\l^2 R_{2424}-2\l R_{1234}+2\l R_{1423} ) . \\
\end{eqnarray*}
Since $R$ has three-nonnegative curvature operator of the second kind, we obtain
\begin{eqnarray*}
 0&\leq & \mathring{R}(\vp_1,\vp_1)+\mathring{R}(\vp_2,\vp_2)+\mathring{R}(\vp_3,\vp_3) \\
 &=& 4(R_{1313}+\l^2 R_{2424})+2(R_{2323}+\l^2 R_{1414})-12\l R_{1234} \\
 && +2\l(-R_{1212}- R_{3434}+ R_{1414}+R_{2323}),
\end{eqnarray*}
where we have used the first Bianchi identity $R_{1342}+R_{1423}=-R_{1234}$.

We then replace $\{e_1,e_2, e_3,e_4\}$ in the above argument by $\{e_2,e_1, e_3,e_4\}$ and replace $\l$ by $-\l$ to get 
\begin{eqnarray*}
 0&\leq & 4(R_{2323}+\l^2 R_{1414})+2(R_{1313}+\l^2 R_{2424} ) -12\l R_{1234}\\
 &&-2\l(-R_{1212}- R_{3434}+ R_{2424}+R_{1313}).
\end{eqnarray*}
Adding the above two inequalities together produces
\begin{eqnarray*}
 0&\leq &  6(R_{1313}+\l^2 R_{1414}+R_{2323}+\l^2 R_{2424} )-24\l R_{1234} \\
 && +2\l(R_{1414}+R_{2323}-R_{2424}-R_{1313}).
\end{eqnarray*}
Replacing $\{e_1,e_2, e_3,e_4\}$ by $\{-e_2,e_1, e_3,e_4\}$ in the above argument yields
\begin{eqnarray*}
 0&\leq &  6(R_{2323}+\l^2 R_{2424}+R_{1313}+\l^2 R_{1414} )-24\l R_{1234} \\
 && +2\l(R_{2424}+R_{1313}-R_{1414}-R_{2323}).
\end{eqnarray*}
We obtain, by adding the above two inequalities together, that
\begin{eqnarray*}
 0&\leq &  12(R_{2323}+\l^2 R_{2424}+R_{1313}+\l^2 R_{1414} )-48\l R_{1234}. 
\end{eqnarray*}
The proof is complete. 
\end{proof}

\begin{lemma}\label{lemma PIC2}
Suppose $n=\dim(V)\geq 4$ and $R \in S^2_B(\Lambda^2(V))$ has positive (respectively, nonnegative) curvature operator of the second kind. Then for any orthonormal four-frame
 $\{e_1, \cdots, e_4\} \subset V$ and any $\l,\mu \in \R$, it holds that 
\begin{equation*}%\label{eq 1}
    R_{1313}+\l^2 R_{1414}+\mu^2 R_{2323}+\l^2 \mu^2 R_{2424} - 6 \l \mu R_{1234} > (\text{respectively, } \geq ) \  0.
\end{equation*}
\end{lemma}

\begin{proof}
Let $\{e_1, \cdots, e_4\}$ be an orthonormal four-frame in $V$. % and we extend it to an orthonormal basis $\{e_1, \cdots, e_n\}$ for $V$.
Given $\l,\mu \in \R$, we define $\vp$ and $\psi$ by
\begin{equation*}
     \vp=e_1 \odot e_3 +\l \mu e_2 \odot e_4
\end{equation*}
and 
\begin{equation*}
    \psi=\mu e_2 \odot e_3 -\l e_1 \odot e_4
\end{equation*}
respectively. 
It's easy to see that both $\vp$ and $\psi$ are traceless symmetric two-tensors on $V$. 

Noticing that the only non-trivial components of $\vp$ are 
\begin{equation*}
    \vp_{13}=\vp_{31}=1 \text{ and }    \vp_{24}=\vp_{42}=\l \mu, 
\end{equation*}
we compute that 
\begin{eqnarray*}
\mathring{R}(\vp,\vp) &=& \sum_{i,j,k,l=1}^n R_{ijkl}\vp_{il}\vp_{jk} \\
&=& R_{1313} +\l \mu R_{1243}+\l \mu R_{1423} \\ 
&& +R_{3131}+\l \mu R_{3241}+\l \mu R_{3421} \\
&& + \l \mu R_{2134}+ \l \mu R_{2314}+ \l^2 \mu^2 R_{2424} \\
&& +\l \mu R_{4132}+\l \mu R_{4312}+\l^2 \mu^2 R_{4242} \\
&=& 2(R_{1313}+\l^2\mu^2 R_{2424}+2\l \mu R_{1423}+2\l \mu R_{1243})
\end{eqnarray*}
Similarly, all components of $\psi$ are trivial except 
\begin{equation*}
    \psi_{23}=\psi_{32}=\mu \text{ and } \psi_{14}=\psi_{41}=-\l,
\end{equation*}
and we calculate 
\begin{eqnarray*}
\mathring{R}(\psi,\psi) &=& \sum_{i,j,k,l=1}^n R_{ijkl}\psi_{il}\psi_{jk} \\
&=& \mu^2 R_{2323} -\l \mu R_{2143}-\l \mu R_{2413} \\ 
&& +\mu^2 R_{3232}-\l \mu R_{3142}-\l \mu R_{3412} \\
&& + \l^2 R_{1414}- \l \mu R_{1234}- \l \mu R_{1324} \\
&& +\l^2 R_{4141}-\l \mu R_{4231}-\l \mu R_{4321} \\
&=& 2( \mu^2R_{2323}+\l^2 R_{1414}-2\l \mu R_{2143}-2\l \mu R_{2413}).
\end{eqnarray*}
Using the assumption that $R$ has nonnegative curvature operator of the second kind, we get 
\begin{eqnarray*}
0 &\leq& \frac 1 2\mathring{R}(\vp,\vp) +\frac 1 2 \mathring{R}(\psi,\psi) \\
&=& R_{1313}+\l^2\mu^2 R_{2424}+2\l \mu R_{1423}+2\l \mu R_{1243} \\
&& + \mu^2R_{2323}+\l^2 R_{1414}-2\l \mu R_{2143}-2\l \mu R_{2413} \\
&=& R_{1313}+\l^2 R_{1414}+\mu^2 R_{2323}+\l^2 \mu^2 R_{2424}-4\l \mu R_{1234} \\
&& +2\l \mu (R_{1423}+R_{1342}) \\
&=& R_{1313}+\l^2 R_{1414}+\mu^2 R_{2323}+\l^2 \mu^2 R_{2424}-6\l \mu R_{1234},
\end{eqnarray*}
where we have used the first Bianchi identity $R_{1423}+R_{1342}=-R_{1234}$ in the last step. 
Similarly, if $R$ has positive curvature curvature operator of the second kind, then the above inequality holds strictly. 
This finishes the proof. 

\end{proof}

It was shown in \cite{CGT21} that four-positive (respectively, four-nonnegative) curvature operator of the second kind implies positive (respectively, nonnegative) isotropic curvature. 
Here we prove a slightly stronger result by improving the coefficient in front of $R_{1234}$ in \eqref{eq PIC} from $2$ to $3$. 
The improvement is necessary to prove Theorem \ref{thm kahler}. Moreover, our proof here is a bit cleaner.
\begin{proposition}\label{prop PIC}
Suppose $n=\dim(V)\geq 4$ and $R\in S^2_B(\Lambda^2(V))$ has four-positive (respectively, four-nonnegative) curvature operator of the second kind. Then for any orthonormal four-frame $\{e_1, \cdots, e_4\} \subset V$, we have
\begin{equation}\label{eq PIC}
     R_{1313}+ R_{1414}+ R_{2323}+  R_{2424} - 3 R_{1234} > (\text{ respectively, } \geq   ) \ 0. 
\end{equation}
\end{proposition}
\begin{proof}
Let $\{e_1, \cdots, e_4\} $ be an orthonormal four-frame in $V$ and define
\begin{eqnarray*}
\vp_1 &=& \frac 1 2 \left( e_1\odot e_1+e_2\odot e_2 -e_3\odot e_3 -e_4 \odot e_4 \right),  \\
\vp_2 &=& \frac 1 2 \left( e_1\odot e_1-e_2\odot e_2 +e_3\odot e_3 -e_4 \odot e_4 \right) , \\
\vp_3 &=& e_1 \odot e_4 -e_2\odot e_3, \\
\vp_4 &=&e_1 \odot e_3 +e_2\odot e_4. 
\end{eqnarray*}
Clearly, $\{\vp_1, \cdots, \vp_4\}$ are traceless symmetric two-tensors on $V$ and they are mutually orthogonal with the same magnitude. 
Direct calculation as before shows that 
\begin{eqnarray*}
\mathring{R}(\vp_1,\vp_1) &=& 2(-R_{1212}-R_{3434}+R_{1313}+R_{2424}+R_{1414}+R_{2323}),\\
\mathring{R}(\vp_2,\vp_2) &=& 2(-R_{1313}-R_{2424}+R_{1212}+R_{3434}+R_{1414}+R_{2323}), \\
\mathring{R}(\vp_3,\vp_3) &=& 2(R_{1414}+R_{2323}-2R_{1234}+2R_{1342} ), \\
\mathring{R}(\vp_4,\vp_4) &=& 2(R_{1313}+R_{2424}-2R_{1234}+2R_{1423} ) .
\end{eqnarray*}
If $R$ has four-nonnegative curvature operator of the second kind, we get 
\begin{eqnarray*}
 0&\leq & \mathring{R}(\vp_1,\vp_1)+\mathring{R}(\vp_2,\vp_2)+\mathring{R}(\vp_3,\vp_3)+\mathring{R}(\vp_3,\vp_3) \\
 &=& 6(R_{1414}+R_{2323})+2(R_{1313}+R_{2424})-12R_{1234}.
\end{eqnarray*}
Replacing $\{e_1, e_2, e_3, e_4\} $ by $\{e_2, -e_1, e_3, e_4\} $ in the above argument yields
\begin{eqnarray*}
 0&\leq & 6(R_{2424}+R_{1313})+2(R_{2323}+R_{1414})-12R_{1234}.
\end{eqnarray*}
We obtain, by adding the above two inequalities together, that 
\begin{eqnarray*}
 0 &\leq & 8(R_{2424}+R_{1313}+R_{2323}+R_{1414})-24 R_{1234}.
\end{eqnarray*}
If $R$ has four-positive curvature operator of the second kind, then all the above inequalities become strict and \eqref{eq PIC} holds strictly. 
The proof is complete. 
\end{proof}

\section{Flat or locally irreducible}

The goal of this section is to investigate the curvature operator of the second kind on product manifolds and prove Theorem \ref{thm irreducible}.

We shall prove a slightly more general statement below, which implies that if the curvature operator of the second kind is $(k(n-k)+1)$-nonnegative for some $1\leq k \leq \frac{n}{2}$, then the manifold cannot split off a $k$-dimensional factor unless it is flat. 
In particular, $n$-manifolds with $n$-nonnegative curvature operator of the second kind must be locally irreducible unless it is flat. 

\begin{proposition}\label{prop split}
Let $(M^n,g)$ be an $n$-dimensional Riemannian manifold. 
Suppose that $M$ is isometric $M_1\times M_2$, where $M_1$ is a $k$-dimensional Riemannian manifold and $M_2$ is an $(n-k)$-dimensional Riemannian manifold. 
If $M$ has $(k(n-k)+1)$-nonnegative curvature operator of the second kind, then $M$ is flat. 
\end{proposition}

\begin{remark}
Proposition \ref{prop split} is sharp in the sense that the assumption in general cannot be weakened to $(k(n-k)+2)$-nonnegative curvature operator of the second kind. 
For instance, when $n=4$ and $k=1$, $\mathbb{S}^3 \times \mathbb{S}^1$ has $5$-nonnegative curvature operator of the second kind (see Example \ref{eg cylinder}), and when $n=4$ and $k=2$, 
$\mathbb{S}^2 \times \mathbb{S}^2$ has $6$-nonnegative curvature operator of the second kind (see Example \ref{eg product of spheres}). 
\end{remark}

\begin{proof}[Proof of Proposition \ref{prop split}]
We first show that both $M_1$ and $M_2$ have nonnegative sectional curvature. 
For $M_1^k$, it suffices to consider the case $k\geq 2$, as one-dimensional Riemannian manifolds are flat. 
Let $\sigma \subset T_pM_1$ be a two-plane spanned by two orthonormal vectors $e_1$ and $e_2$. 
We extend $\{e_1,e_2\}$ to an orthonormal basis $\{e_1, \cdots e_k\}$ of $T_{p}M_1$. Let $\{e_{k+1}, \cdots, e_n\}$ be an orthonormal basis of $T_qM_2$. Then $\{e_1, \cdots, e_n\}$ is an orthonormal basis of $T_{(p,q)}M$.

Define
\begin{equation*}
    \vp_{ip}=\frac{1}{\sqrt{2}} e_i \odot e_p \text{   for } 1\leq i \leq k, \ k+1 \leq p \leq n,
\end{equation*}
and 
\begin{equation*}
    \xi=\frac{1}{\sqrt{2}}e_1 \odot e_2.
\end{equation*}
It's easy to verify that $\{\vp_{ip}\}_{1\leq i \leq k, \ k+1 \leq p \leq n} \cup \{\xi \}$ is an orthonormal set of dimension $k(n-k)+1$ in $S^2_0(\Lambda^2 T_{(p,q)}M)$. 
Since $M$ has $(k(n-k)+1)$-nonnegative curvature operator of the second kind, we have 
\begin{equation*}
    \mathring{R}(\xi,\xi) +\sum_{1\leq i \leq k, \ k+1 \leq p \leq n} \mathring{R}(\vp_{ip},\vp_{ip}) \geq 0.
\end{equation*}
Notice that 
\begin{equation}
    \mathring{R}(\vp_{ip},\vp_{ip}) =R_{ipip} =0 \text{ for } 1\leq i \leq k, \ k+1 \leq p \leq n. 
\end{equation}
due to the product structure.
We therefore have $\mathring{R}(\xi,\xi) =R_{1212} \geq 0$. Since $\sigma$ is arbitrary, we conclude that $M_1$ has nonnegative sectional curvature. $M_2$ is also nonnegatively curved by a similar argument.

The next step is to show that both $M_1$ and $M_2$ have vanishing scalar curvature.
Let $\{e_1, \cdots, e_n\}$ be an orthonormal basis of $T_{(p,q)}M$ with 
$e_1, \cdots e_k \in  T_{p}M_1$ and $e_{k+1}, \cdots, e_n \in T_qM_2$. 
Define $\vp_{ip}$ as before and let 
\begin{equation*}
    \psi=\frac{1}{2\sqrt{nk(n-k)}}\left((n-k) \sum_{i=1}^k e_i \odot e_i - k \sum_{p=k+1}^n e_p \odot e_p \right).
\end{equation*}
One easily verifies that $\{\vp_{ip}\}_{1\leq i \leq k, \ k+1 \leq p \leq n} \cup \{\psi \}$ is an orthonormal set of dimension $k(n-k)+1$ in $S^2_0(\Lambda^2 T_{(p,q)}M)$. 
Since $M$ has $(k(n-k)+1)$-nonnegative curvature operator of the second kind, we have 
\begin{equation*}
    \mathring{R}(\psi,\psi) = \mathring{R}(\psi,\psi) +\sum_{1\leq i \leq k, \ k+1 \leq p \leq n} \mathring{R}(\vp_{ip},\vp_{ip}) \geq 0.
\end{equation*}
We calculate that
\begin{eqnarray*}
0 &\leq & n k(n-k) \mathring{R}(\psi,\psi) \\
&=& (n-k)^2 \sum_{1\leq i,j \leq k}R_{ijji} 
-k(n-k) \sum_{1\leq i \leq k, k+1 \leq p \leq k}R_{ippi} \\
&& -k(n-k)\sum_{1\leq j \leq k, k+1 \leq q \leq n}R_{qjjq} +k^2 \sum_{k+1\leq p,q \leq n}R_{pqqp} \\
&=& -(n-k)^2 S_1(p)-k^2 S_2(q),
\end{eqnarray*}
where $S_i$ denotes the scalar curvature of $M_i$ for $i=1,2$.
The above inequality forces $S_1(p) \leq 0$ and $S_2(q) \leq 0$. 
On the other hand, we must have $S_1(p)\geq 0$ and $S_2(q) \geq 0$, as we have showed that both $M_1$ and $M_2$ are nonnegatively curved. Therefore both $M_1$ and $M_2$ must be scalar flat, and thus flat in view of the nonnegativity of their sectional curvatures. 
Hence $M$ is flat.

\end{proof}

\begin{proof}[Proof of Theorem \ref{thm irreducible}]
Since $n$-nonnegative curvature operator of the second kind implies $(k(n-k)+1)$-nonnegative curvature operator of the second kind for all $1\leq k \leq n-1$, the universal cover of $M$ cannot split as a product of Riemannian manifolds of lower dimensions unless it is flat. 
Thus $M$ is either flat or locally irreducible.  
\end{proof}

\section{K\"ahler Manifolds}

In this section, we prove Theorem \ref{thm kahler}. We indeed prove it under a weaker assumption.
\begin{theorem}\label{thm kahler stronger}
Let $(M^{2m},g)$ be a K\"ahler manifold of complex dimension $m\geq 2$. Suppose that there exists $\beta >1$ such that for any $p\in M$ and any orthonormal four-frame $\{e_1, \cdots, e_4\} \subset T_pM$, it holds that 
\begin{equation}\label{eq 4.1}
    R_{1313}+R_{1414}+R_{2323}+R_{2424}-2\beta R_{1234} \geq 0.
\end{equation}
Then $M$ must be flat if $m\geq 3$. 
If $m=2$, then $M$ is conformally flat, which means locally it is either flat or the product of two complex curves with constant curvature of opposite values.
\end{theorem}

The curvature assumption in Theorem \ref{thm kahler stronger} is optimal in the sense that there are many K\"ahler manifolds, such as compact Hermitian symmetric spaces, satisfying \eqref{eq 4.1} with $\beta =1$, namely they have nonnegative isotropic curvature.

\begin{proof}[Proof of Theorem \ref{thm kahler stronger}]
The proof here is inspired by Brendle's argument \cite[Proposition 9.18]{Brendle10book} to prove that K\"ahler-Einstein manifolds with nonnegative isotropic curvature must have positive orthogonal bisectional curvature if it is not flat. The assumption $\beta>1$ in some sense compensates the Einstein condition.
Most of the calculation below is exactly the same as in Brendle's argument \cite[Proposition 9.18]{Brendle10book}.

Recall that on a K\"ahler manifold there exists a section $J$ of the endomorphism bundle $\text{End}(TM)$ with the following properties: 
\begin{enumerate}
    \item $J$ is parallel;
    \item for each point $p\in M$, we have $J^2=-\id$ and $g(X,Y)=g(JX,JY)$ for all $X,Y\in T_pM$;
    \item the Riemann curvature tensor satisfies 
    $$R(X,Y,Z,W)=R(X,Y,JZ,JW)$$
    for all $X,Y,Z,W\in T_pM$.
\end{enumerate}

\textbf{Claim A:} $M$ has vanishing orthogonal bisectional curvature, namely for all unit vectors $X,Y \in T_pM$ satisfying $g(X,Y)=g(X,JY)=0$, it holds that
\begin{equation}\label{eq 4.3}
    R(X,JX,Y,JY)=0.
\end{equation}

\textit{Proof of Claim A}: Consider two unit vectors $X,Y \in T_pM$ with $g(X,Y)=g(JX,Y)=0$. Notice that by (2) we have 
\begin{equation*}
    g(JX, Y)=g(J^2X,JY)=-g(X,JY) =0
\end{equation*} 
and $g(JX,JY)=g(X,Y)=0$. 
Thus $\{X,JX,Y,JY\}$ is an orthonormal four-frame in $T_pM$.

Applying \eqref{eq 4.1} to the orthonormal four-frame $\{X,JX,Y,JY\}$ produces 
\begin{eqnarray*}
 && R(X,Y,X,Y)+R(JX,Y,JX,Y)\\
 &&+ R(X,JY,X,JY) +R(JX,JY,JX,JY) \\
 &\geq& 2\beta R(X,JX,Y,JY).
\end{eqnarray*}
Notice that \eqref{eq 4.1} implies \begin{equation}\label{eq 4.2}
    R_{1313}+R_{1414}+R_{2323}+R_{2424}\geq 0
\end{equation}
for any orthonormal four-frame $\{e_1, \cdots, e_4\} \subset T_pM$.
Applying \eqref{eq 4.2} to the orthonormal four-frame $\{X,JX,Y,JY\}$ yields
\begin{align*}
   & R(X,Y,X,Y)+R(JX,Y,JX,Y)\\
 &+ R(X,JY,X,JY)+R(JX,JY,JX,JY) \geq 0.
\end{align*}

On the other hand, it follows from the first Bianchi identity and (3) that 
\begin{eqnarray*}
 2 R(X,JX,Y,JY)& = & R(X,Y,X,Y)+R(JX,Y,JX,Y)\\
 &&+ R(X,JY,X,JY)+R(JX,JY,JX,JY).
\end{eqnarray*}
Therefore, we obtain that 
\begin{equation*}
    R(X,JX,Y,JY) \geq 0
\end{equation*}
and 
\begin{equation*}
    2 R(X,JX,Y,JY) \geq 2\beta R(X,JX,Y,JY) 
\end{equation*}
Since $\beta >1$, we must have 
\begin{equation*}%\label{eq 4.3}
    R(X,JX,Y,JY)=0.
\end{equation*}
This finishes the proof of Claim A.

\textbf{Claim B:} $M$ has vanishing scalar curvature. 

\textit{Proof of Claim B:} 
Let $Z,W \in T_p M$ be two unit vectors satisfying $g(Z,W)=g(Z,JW)=0$. 
Applying \eqref{eq 4.3} to the vectors $X=\frac{1}{\sqrt{2}}(Z+W)$ and $Y=\frac{1}{\sqrt{2}}(Z-W)$ yields 
\begin{eqnarray*}
0 &=& R(Z+W, JZ+JW, Z-W, JZ-JW) \\
&=& R(Z,JZ,Z,JZ)+R(Z,JZ,W,JW) \\
&& -R(Z,JZ,Z,JW )-R(Z,JZ,W,JZ) \\
&& +R(W,JW, Z,JZ)+R(W,JW,W,JW)\\
&& -R(W,JW, Z,JW)-R(W,JW,W,JZ) \\
&& + R(Z,JW, Z,JZ)+R(Z,JW,W,JW) \\
&& -R(Z,JW, Z,JW)-R(Z,JW,W,JZ) \\
&& +R(W,JZ,Z,JZ)+R(W,JZ,W,JW) \\
&& -R(W,JZ,Z,JW) -R(W,JZ, W, JZ) \\
&=& R(Z,JZ,Z,JZ)+R(W,JW,W,JW) \\
&&+2R(Z,JZ,W,JW) -4R(Z,JW,Z,JW)
\end{eqnarray*}
Similarly, we apply \eqref{eq 4.3} to the vectors $X=\frac{1}{\sqrt{2}}(Z+JW)$ and $Y=\frac{1}{\sqrt{2}}(Z-JW)$ and obtain 
\begin{eqnarray*}
0 &=& R(Z+W, JZ+JW, Z-W, JZ-JW) \\
&=& R(Z,JZ,Z,JZ)+R(W,JW,W,JW) \\
&&+2R(Z,JZ,W,JW) -4R(Z,W,Z,W).
\end{eqnarray*}
Averaging the above two equations yields,
\begin{eqnarray*}
0 &=& R(Z,JZ,Z,JZ)+R(W,JW,W,JW) \\
&& +2R(Z,JZ,W,JW)-2R(Z,JW,Z,JW)-2R(Z,W,Z,W).
\end{eqnarray*}
Noticing that by the first Bianchi identity and (3), we have
\begin{eqnarray*}
&& R(Z,JZ,W,JW)-R(Z,JW,Z,JW)-R(Z,W,Z,W) \\
&=& R(Z,JZ,W,JW)+R(Z,JW,JZ,W) +R(Z,W,JW,JZ) \\
&=& 0.
\end{eqnarray*}
Therefore, we have proved that 
\begin{equation}\label{eq 4.4}
    R(Z,JZ,Z,JZ)+R(W,JW,W,JW) =0
\end{equation}
for all unit vectors satisfying $g(Z,W)=g(Z,JW)=0$.
It follows form \eqref{eq 4.3} and \eqref{eq 4.4} that the scalar curvature of $M$ vanishes, as the scalar curvature of a K\"ahler manifold is given by 
\begin{equation*}\label{eq scalar curvature}
    S=2\sum_{i,j=1}^m R(e_i,Je_i,e_j,Je_j),
\end{equation*}
where $\{e_1,\cdots, e_m, Je_1, \cdots, Je_m\}$ is an orthonormal basis of $T_pM$.
This proves Claim B.

We continue to prove Theorem \ref{thm kahler stronger}. Clearly \eqref{eq 4.1} implies that $M$ has nonnegative isotropic curvature. 
Therefore, Theorem \ref{thm kahler stronger} follows from a result of Micallef and Wang \cite[Proposition 2.5]{MM93}, which asserts that a Riemannian manifold of real dimension $n\geq 4$ with nonnegative isotropic curvature and vanishing scalar curvature must be flat if $n\geq 5$ and must have vanishing Weyl tensor if $n=4$. 

Alternatively, one can argue as follows after proving Claim A. Note that \eqref{eq 4.3} implies that $M$ has vanishing orthogonal Ricci curvature $\Ric^\perp \equiv 0$ ( $\Ric^\perp$ was introduced by Ni and Zheng \cite{NZ18} in the study of comparison theorems for K\"ahler manifolds).
The conclusion then follows from the classification of $\Ric^\perp$-flat K\"ahler manifold by Ni, Wang and Zheng \cite[Theorem 6.1]{NWZ21}.
\end{proof}

\begin{proof}[Proof of Theorem \ref{thm kahler}]
Since $M$ has four-nonnegative curvature operator of the second kind, it satisfies \eqref{eq 4.1} with $\beta=\frac 3 2$ by Proposition \ref{prop PIC}.
By Theorem \ref{thm kahler stronger}, $M$ must be flat if $m\geq 3$. If $m=2$, one also uses Theorem \ref{thm irreducible} to rule out the product of two complex curves with constant curvature of opposite values. Hence $M$ must be flat. 
\end{proof}

\section{Proof of Theorem \ref{thm Nishikawa}}

\begin{proof}[Proof of Theorem \ref{thm Nishikawa}]
Suppose that $M$ is non-flat. 
Let $\widetilde{M}$ be the universal cover of $M$ with the lifted metric $\tilde{g}$. 
Then $(\widetilde{M},\tilde{g})$ also has three-nonnegative curvature operator of the second kind and nonnegative Ricci curvature by part (3) of Proposition \ref{prop algebraic implications}. 
By the Cheeger-Gromoll splitting theorem \cite{CG72}, $\widetilde{M}$ is isometric to a product of the form
$N^{n-k} \times \R^k$, where $N$ is compact. 
By Theorem \ref{thm irreducible}, we know that $M$ is locally irreducible, which implies that $k=0$. 
Hence $\widetilde{M}$ is compact. 

On the other hand, it follows from part (4) of Proposition \ref{prop algebraic implications} that $\widetilde{M}$ has weakly PIC1. 
By evoking the classification of closed simply-connected irreducible Riemannian manifold with weakly PIC1 (see for example \cite[Theorem 9.33]{Brendle10book}), we conclude that one of the following statements holds:
\begin{enumerate}
    \item $\widetilde{M}$ is diffeomorphic to $\mathbb{S}^n$;
    \item $n=2m$ and $\widetilde{M}$ is a K\"ahler manifold; %biholomorphic to $\mathbb{C} P^m$; 
    \item $\widetilde{M}$ is isometric to a compact irreducible symmetric space.
\end{enumerate}
Notice that (2) cannot occur in view of Theorem \ref{thm kahler}.
Also, (3) cannot occur when $n\leq 4$ because 
%the only compact irreducible symmetric space other than the round sphere is 
$\mathbb{CP}^2$ does not have three-nonnegative curvature operator of the second kind. 

Hence we can conclude that $M$ is either flat, or diffeomorphic to a spherical space form, or isometric to a quotient of a compact irreducible symmetric space. 
In particular, $M$ is diffeomorphic to a Riemannian locally symmetric space.
The proof of Theorem \ref{thm Nishikawa} is now complete. 
\end{proof}

\section*{Acknowledgments}
The author is grateful to Professors Xiaodong Cao, Matthew Gursky and Hung Tran for some helpful discussions. 
He also would like to thank the anonymous referee for catching some typos and making comments and suggestions that improved the readability of this paper.

\bibliographystyle{alpha}
\bibliography{ref}

\begin{thebibliography}{NPWW22}

\bibitem[AN09]{AN09}
Ben Andrews and Huy Nguyen.
\newblock Four-manifolds with 1/4-pinched flag curvatures.
\newblock {\em Asian J. Math.}, 13(2):251--270, 2009.

\bibitem[BE69]{BE69}
M.~Berger and D.~Ebin.
\newblock Some decompositions of the space of symmetric tensors on a
  {R}iemannian manifold.
\newblock {\em J. Differential Geometry}, 3:379--392, 1969.

\bibitem[Bes08]{Besse08}
Arthur~L. Besse.
\newblock {\em Einstein manifolds}.
\newblock Classics in Mathematics. Springer-Verlag, Berlin, 2008.
\newblock Reprint of the 1987 edition.

\bibitem[BK78]{BK78}
Jean-Pierre Bourguignon and Hermann Karcher.
\newblock Curvature operators: pinching estimates and geometric examples.
\newblock {\em Ann. Sci. \'{E}cole Norm. Sup. (4)}, 11(1):71--92, 1978.

\bibitem[Bor60]{Borel60}
Armand Borel.
\newblock On the curvature tensor of the {H}ermitian symmetric manifolds.
\newblock {\em Ann. of Math. (2)}, 71:508--521, 1960.

\bibitem[Bre08]{Brendle08}
Simon Brendle.
\newblock A general convergence result for the {R}icci flow in higher
  dimensions.
\newblock {\em Duke Math. J.}, 145(3):585--601, 2008.

\bibitem[Bre10]{Brendle10book}
Simon Brendle.
\newblock {\em Ricci flow and the sphere theorem}, volume 111 of {\em Graduate
  Studies in Mathematics}.
\newblock American Mathematical Society, Providence, RI, 2010.

\bibitem[BS08]{BS08}
Simon Brendle and Richard~M. Schoen.
\newblock Classification of manifolds with weakly {$1/4$}-pinched curvatures.
\newblock {\em Acta Math.}, 200(1):1--13, 2008.

\bibitem[BS09]{BS09}
Simon Brendle and Richard Schoen.
\newblock Manifolds with {$1/4$}-pinched curvature are space forms.
\newblock {\em J. Amer. Math. Soc.}, 22(1):287--307, 2009.

\bibitem[BS11]{BS11}
Simon Brendle and Richard Schoen.
\newblock Curvature, sphere theorems, and the {R}icci flow.
\newblock {\em Bull. Amer. Math. Soc. (N.S.)}, 48(1):1--32, 2011.

\bibitem[BW08]{BW08}
Christoph B{\"o}hm and Burkhard Wilking.
\newblock Manifolds with positive curvature operators are space forms.
\newblock {\em Ann. of Math. (2)}, 167(3):1079--1097, 2008.

\bibitem[CG72]{CG72}
Jeff Cheeger and Detlef Gromoll.
\newblock The splitting theorem for manifolds of nonnegative {R}icci curvature.
\newblock {\em J. Differential Geometry}, 6:119--128, 1971/72.

\bibitem[CGT21]{CGT21}
Xiaodong Cao, Matthew~J. Gursky, and Hung Tran.
\newblock Curvature of the second kind and a conjecture of {N}ishikawa.
\newblock {\em Comment. Math. Helv., to appear, arXiv:2112.01212}, 2021.

\bibitem[Che91]{Chen91}
Haiwen Chen.
\newblock Pointwise {$\frac14$}-pinched {$4$}-manifolds.
\newblock {\em Ann. Global Anal. Geom.}, 9(2):161--176, 1991.

\bibitem[CV60]{CV60}
Eugenio Calabi and Edoardo Vesentini.
\newblock On compact, locally symmetric {K}\"{a}hler manifolds.
\newblock {\em Ann. of Math. (2)}, 71:472--507, 1960.

\bibitem[GM75]{GM75}
S.~Gallot and D.~Meyer.
\newblock Op\'{e}rateur de courbure et laplacien des formes diff\'{e}rentielles
  d'une vari\'{e}t\'{e} riemannienne.
\newblock {\em J. Math. Pures Appl. (9)}, 54(3):259--284, 1975.

\bibitem[Ham82]{Hamilton82}
Richard~S. Hamilton.
\newblock Three-manifolds with positive {R}icci curvature.
\newblock {\em J. Differential Geom.}, 17(2):255--306, 1982.

\bibitem[Ham86]{hamilton86}
Richard~S. Hamilton.
\newblock Four-manifolds with positive curvature operator.
\newblock {\em J. Differential Geom.}, 24(2):153--179, 1986.

\bibitem[Kas93]{Kashiwada93}
Toyoko Kashiwada.
\newblock On the curvature operator of the second kind.
\newblock {\em Natur. Sci. Rep. Ochanomizu Univ.}, 44(2):69--73, 1993.

\bibitem[Koi79a]{Koiso79a}
Norihito Koiso.
\newblock A decomposition of the space {${\mathcal{M}}$} of {R}iemannian
  metrics on a manifold.
\newblock {\em Osaka Math. J.}, 16(2):423--429, 1979.

\bibitem[Koi79b]{Koiso79b}
Norihito Koiso.
\newblock On the second derivative of the total scalar curvature.
\newblock {\em Osaka Math. J.}, 16(2):413--421, 1979.

\bibitem[Li22a]{Li22Kahler}
Xiaolong Li.
\newblock K\"ahler manifolds and the curvature operator of the second kind.
\newblock {\em arXiv:2208.14505}, 2022.

\bibitem[Li22b]{Li22PAMS}
Xiaolong Li.
\newblock K\"ahler surfaces with six-positive curvature operator of the second
  kind.
\newblock {\em Proc. Amer. Math. Soc., to appear, arXiv:2207.00520}, 2022.

\bibitem[Li22c]{Li22JGA}
Xiaolong Li.
\newblock Manifolds with {$4\frac 12$}-{P}ositive {C}urvature {O}perator of the
  {S}econd {K}ind.
\newblock {\em J. Geom. Anal.}, 32(11):281, 2022.

\bibitem[Li22d]{Li22product}
Xiaolong Li.
\newblock Product manifolds and the curvature operator of the second kind.
\newblock {\em arXiv:2209.02119}, 2022.

\bibitem[Mey71]{Meyer71}
Daniel Meyer.
\newblock Sur les vari\'{e}t\'{e}s riemanniennes \`a op\'{e}rateur de courbure
  positif.
\newblock {\em C. R. Acad. Sci. Paris S\'{e}r. A-B}, 272:A482--A485, 1971.

\bibitem[MM88]{MM88}
Mario~J. Micallef and John~Douglas Moore.
\newblock Minimal two-spheres and the topology of manifolds with positive
  curvature on totally isotropic two-planes.
\newblock {\em Ann. of Math. (2)}, 127(1):199--227, 1988.

\bibitem[MRS20]{MRS20}
Josef Mike\v{s}, Vladimir Rovenski, and Sergey~E. Stepanov.
\newblock An example of {L}ichnerowicz-type {L}aplacian.
\newblock {\em Ann. Global Anal. Geom.}, 58(1):19--34, 2020.

\bibitem[MW93]{MM93}
Mario~J. Micallef and McKenzie~Y. Wang.
\newblock Metrics with nonnegative isotropic curvature.
\newblock {\em Duke Math. J.}, 72(3):649--672, 1993.

\bibitem[Nis86]{Nishikawa86}
Seiki Nishikawa.
\newblock On deformation of {R}iemannian metrics and manifolds with positive
  curvature operator.
\newblock In {\em Curvature and topology of {R}iemannian manifolds ({K}atata,
  1985)}, volume 1201 of {\em Lecture Notes in Math.}, pages 202--211.
  Springer, Berlin, 1986.

\bibitem[NPW22]{NPW22}
Jan Nienhaus, Peter Petersen, and Matthias Wink.
\newblock Betti numbers and the curvature operator of the second kind.
\newblock {\em arXiv:2206.14218}, 2022.

\bibitem[NPWW22]{NPWW22}
Jan Nienhaus, Peter Petersen, Matthias Wink, and William Wylie.
\newblock Holonomy restrictions from the curvature operator of the second kind.
\newblock {\em arXiv:2208.13820}, 2022.

\bibitem[NW07a]{NWolfson07}
Lei Ni and Jon Wolfson.
\newblock Positive complex sectional curvature, {R}icci flow and the
  differential sphere theorem.
\newblock {\em arXiv:0706.0332}, 2007.

\bibitem[NW07b]{NWu07}
Lei Ni and Baoqiang Wu.
\newblock Complete manifolds with nonnegative curvature operator.
\newblock {\em Proc. Amer. Math. Soc.}, 135(9):3021--3028, 2007.

\bibitem[NW10]{NW10}
Lei Ni and Burkhard Wilking.
\newblock Manifolds with {$1/4$}-pinched flag curvature.
\newblock {\em Geom. Funct. Anal.}, 20(2):571--591, 2010.

\bibitem[NWZ21]{NWZ21}
Lei Ni, Qingsong Wang, and Fangyang Zheng.
\newblock Manifolds with positive orthogonal {R}icci curvature.
\newblock {\em Amer. J. Math.}, 143(3):833--857, 2021.

\bibitem[NZ18]{NZ18}
Lei Ni and Fangyang Zheng.
\newblock Comparison and vanishing theorems for {K}\"{a}hler manifolds.
\newblock {\em Calc. Var. Partial Differential Equations}, 57(6):Art. 151, 31,
  2018.

\bibitem[OT79]{OT79}
Koichi Ogiue and Shun-ichi Tachibana.
\newblock Les vari\'{e}t\'{e}s riemanniennes dont l'op\'{e}rateur de courbure
  restreint est positif sont des sph\`eres d'homologie r\'{e}elle.
\newblock {\em C. R. Acad. Sci. Paris S\'{e}r. A-B}, 289(1):A29--A30, 1979.

\bibitem[Tac74]{Tachibana74}
Shun-ichi Tachibana.
\newblock A theorem of {R}iemannian manifolds of positive curvature operator.
\newblock {\em Proc. Japan Acad.}, 50:301--302, 1974.

\bibitem[Wil07]{Wilking07}
Burkhard Wilking.
\newblock Nonnegatively and positively curved manifolds.
\newblock In {\em Surveys in differential geometry. {V}ol. {XI}}, volume~11 of
  {\em Surv. Differ. Geom.}, pages 25--62. Int. Press, Somerville, MA, 2007.

\bibitem[Wil13]{Wilking13}
Burkhard Wilking.
\newblock A {L}ie algebraic approach to {R}icci flow invariant curvature
  conditions and {H}arnack inequalities.
\newblock {\em J. Reine Angew. Math.}, 679:223--247, 2013.

\end{thebibliography}

\end{document}